\documentclass[10pt]{amsart}

\usepackage{enumerate}
\usepackage{amsmath,amssymb}

\newtheorem{theorem}{Theorem}
\newtheorem{lemma}[theorem]{Lemma}
\newtheorem{corollary}[theorem]{Corollary}

\def\bC{{\mathbb{C}}}

\def\cX{{\mathcal{X}}}
\def\cB{{\mathcal{B}}}
\def\cM{{\mathcal{M}}}
\def\cN{{\mathcal{N}}}
\def\cP{{\mathcal{P}}}

\def\rank{{\operatorname{rank}\,}}

\begin{document}

\title[Distance from a matrix to nilpotents]{On the distance from a matrix to nilpotents}

\author[M. Mori]{Michiya Mori}

\address{Graduate School of Mathematical Sciences, The University of Tokyo, 3-8-1 Komaba, Meguro-ku, Tokyo, 153-8914, Japan; Interdisciplinary Theoretical and Mathematical Sciences Program (iTHEMS), RIKEN, 2-1 Hirosawa, Wako, Saitama, 351-0198, Japan.}
\email{mmori@ms.u-tokyo.ac.jp}

\thanks{The author is supported by JSPS KAKENHI Grant Number 22K13934.}


\date{}

\begin{abstract}
We prove that the distance from an $n\times n$ complex matrix $M$ to the set of nilpotents is at least $\frac{1}{2}\sec\frac{\pi}{n+2}$ if there is a nonzero projection $P$ such that $PMP=M$ and $M^*M\geq P$.
In the particular case where $M$ equals $P$, this verifies a conjecture by G.W. MacDonald in 1995.
We also confirm a related conjecture in D.A. Herrero's book.
\end{abstract}

\maketitle
\thispagestyle{empty}

\section{Introduction}
Throughout this note, $n$ is a positive integer. 
Let $\cM_n$ denote the set of $n\times n$ complex matrices, and let $\cN_n$ denote the subset of $n\times n$ nilpotent matrices. 
In the current paragraph, we assume that $A, B\in \cM_n$.
We endow $\cM_n$ with the usual distance $d$ that comes from the matrix (operator) norm, that is, $d(A, B)=\lVert A-B\rVert$.
Let $\nu(A)$ stand for the distance $\inf_{N\in \cN_n} d(A, N)$ between $A$ and $\cN_n$. 
Some properties of $\nu$ can be derived easily. 
For example, the inequality $\lvert\nu(A)-\nu(B)\rvert \leq \lVert A-B\rVert$ holds.
We have $\nu(\lambda A)=\lvert \lambda\rvert \nu(A)$ for every $\lambda\in \mathbb{C}$.
If $U$ is a unitary in $\cM_n$, then $\nu(A) =\nu(UAU^*)$. 
Let the symbol $\rho(A)$ denote the spectral radius of $A$. 
Since every matrix is unitarily equivalent to an upper triangular matrix (Schur triangulation), we obtain $\nu(A)\leq \rho(A)$. 

However, in general, it is apparently very hard to determine the precise value of $\nu(A)$ for a matrix $A$.
The following question concerning $\nu$ has long been studied in the literature: 
What is the infimum of $\nu(P)$ over all nonzero projections (hermitian idempotents) $P$ in $\cM_n$?
For $0\leq m\leq n$, let $\cP_{n,m}$ denote the set of projections in $\cM_n$ of rank $m$.
To the best of our knowledge, the first research on the above question was given in \cite[Section 1]{Hed} a half century ago, in which Hedlund writes that the determination of the precise value seems difficult.
More than 20 years later, MacDonald proved that $\nu(P)=\frac{1}{2}\sec\frac{\pi}{n+2}$ for $P\in \cP_{n,1}$, and  conjectured that $\nu(P)\geq \frac{1}{2}\sec\frac{\pi}{n+2}$ holds for every nonzero projection $P$ in $\cM_n$ \cite{M1}.
He verified it for $n\leq 4$ \cite[Lemma 3.3]{M2}.
In the recent paper \cite{C}, Cramer conjectured that $\nu(P)=\frac{1}{2}\sec\frac{\pi}{\frac{n}{m} + 2}$ for every $1\leq m\leq n$ and $P\in \cP_{n, m}$, and proved that this is the case when $m=n-1$. 
For more information around this topic, see \cite[Chapter 2]{Her5}, \cite{C}, \cite{M2}.
The main purpose of this note is to verify MacDonald's conjecture.

We also consider a similar problem concerning operators on a (possibly infinite-dimensional) complex Hilbert space $H$.
Let $\cB(H)$ denote the set of bounded linear operators on $H$.
For $A\in \cB(H)$, let $\nu_n(A)$ stand for the distance between $A$ and the set $\{N\in \cB(H)\mid N^n=0\}$. 
We show that $\nu_n(P)\geq \frac{1}{2}\sec\frac{\pi}{n+2}$ for every nonzero projection $P\in \cB(H)$. 
This gives a proof to a conjecture posed by Herrero in the first part of \cite[Conjecture 2.16]{Her5}. Note that the latter part of \cite[Conjecture 2.16]{Her5} is not correct, see \cite[page 849]{M1}. 

In fact, below we prove Theorems \ref{t1} and \ref{t2}, which are much more general than MacDonald's conjecture and Herrero's respectively.
\section{Proofs}
When $A$ and $B$ are hermitian matrices (operators), the symbol $A\leq B$ means that $B-A$ is positive semidefinite.
If $0\leq A$, then $A^{1/2}$ stands for the positive square root of $A$.
For a projection $P$ in $\cM_n$, we define $P^\perp:=I_n-P$.
Let $\Pi_n$ denote the set of $(n+1)$-tuples of projections $(P_0, P_1, \ldots, P_n)$ satisfying $P_0\leq P_1\leq \cdots \leq P_n$ and $P_k\in \cP_{n, k}$ for every $0\leq k\leq n$.
Recall that a matrix $N\in \cM_n$ is nilpotent if and only if $N^n=0$, and this is equivalent to the existence of $(P_0, P_1, \ldots, P_n)\in \Pi_n$ such that $P_{k-1}^\perp NP_k=0$ for every $k$.
Thus, Power's version \cite[Lemma]{P} of Arveson's distance formula \cite{A} gives the following equation (see also \cite[Theorem 2]{M1}): For $A\in \cM_n$,
\[
\nu(A)=\inf\{\max_{1\leq k\leq n}\lVert P_{k-1}^\perp AP_k\rVert\mid (P_0, P_1, \ldots, P_n)\in \Pi_n\}.
\] 

\begin{lemma}\label{l0}
Let $X, Y\in \cM_n$. Then $\lVert XY\rVert= \lVert (X^*X)^{1/2}(YY^*)^{1/2}\rVert$.
\end{lemma}
\begin{proof}
By considering polar decompositions of $X$ and $Y$, we get partial isometries $V, W$ such that 
$XY =V(X^*X)^{1/2}(YY^*)^{1/2}W$ and $(X^*X)^{1/2}(YY^*)^{1/2}=V^*XYW^*$. 
This leads to the desired equation.
\end{proof}

\begin{lemma}\label{l1}
Assume that $0\leq A\in \cM_n$. Let  $\Pi(A)$ denote the set of $(n+1)$-tuples of matrices $(A_0, A_1, \ldots, A_n)$ satisfying $0=A_0\leq A_1\leq \cdots\leq A_n=A$ and $\rank (A_k-A_{k-1})\leq 1$ for every $1\leq k\leq n$.
Then 
\[
\nu(A)=\alpha:=\inf\{\max_{1\leq k\leq n}\lVert (A-A_{k-1})^{1/2}A_k^{1/2}\rVert\mid(A_0, A_1, \ldots, A_n)\in \Pi(A)\}.
\] 
\end{lemma}
\begin{proof}
Let $(P_0, P_1, \ldots, P_n)\in \Pi_n$.
For each $k$, Lemma \ref{l0} with $X=P_{k-1}^\perp A^{1/2}$ and $Y=A^{1/2}P_k$ shows
$\lVert P_{k-1}^\perp AP_k\rVert= \lVert (A-A_{k-1})^{1/2}A_k^{1/2}\rVert$,
where $A_k=A^{1/2}P_kA^{1/2}$. 
It is easy to see $(A_0, A_1, \ldots, A_n)\in \Pi(A)$, thus we obtain $\alpha\leq \nu(A)$.

To get the other inequality, assume that $(A_0, A_1, \ldots, A_n)\in \Pi(A)$. 
Let $\varepsilon>0$. 
Take a matrix $C_k\geq 0$ of rank one that is close to $A_k-A_{k-1}$ for each $k$ such that $B=C_1+C_2+\cdots +C_n$ has rank $n$.  
We may assume $\lVert B-A\rVert<\varepsilon$ and $\lVert (B-B_{k-1})^{1/2}B_k^{1/2}  -  (A-A_{k-1})^{1/2}A_k^{1/2}\rVert<\varepsilon$, where $B_0 =0$ and $B_k=C_1+C_2+\cdots+ C_k$,  for every $k\in \{1, \ldots, n\}$.
Thus $\lVert  (A-A_{k-1})^{1/2}A_k^{1/2}\rVert>\lVert (B-B_{k-1})^{1/2} B_k^{1/2}\rVert -\varepsilon$.
Since $B$ has rank $n$,  $B^{1/2}$ is invertible. 
Set $P_k:= B^{-1/2}B_kB^{-1/2}$ for each $k$. 
Then we have $0=P_0\leq P_1\leq \cdots\leq P_n=I_n$ and $\rank (P_k-P_{k-1})\leq 1$ for every $1\leq k\leq n$.
It follows that $\rank P_k\leq k$ and $\rank (I_n -P_k)\leq n-k$, which imply $(P_0, P_1, \ldots, P_n)\in \Pi_n$.
Lemma \ref{l0} implies $\lVert (B-B_{k-1})^{1/2}B_k^{1/2}\rVert =\lVert P_{k-1}^\perp BP_k\rVert$ for each $k$.
By
$\lvert \nu(B)-\nu(A)\lvert\leq \lVert B-A\lVert <\varepsilon$, we obtain 
\[
\max_{1\leq k\leq n}\lVert (A-A_{k-1})^{1/2}A_k^{1/2} \rVert > \max_{1\leq k\leq n}\lVert P_{k-1}^\perp BP_k\rVert-\varepsilon \geq \nu(B) -\varepsilon > \nu(A)-2\varepsilon.
\]
Since $\varepsilon>0$ is arbitrary, we get $\alpha\geq \nu(A)$.
\end{proof}

Essentially the same proof gives 
\begin{lemma}\label{l2}
Let $0\leq A\in \cM_n$ and $X\in \cM_n$. 
Then 
\[
\nu(A^{1/2}XA^{1/2})=\inf\{\max_{1\leq k\leq n}\lVert (A-A_{k-1})^{1/2}(XA_kX^*)^{1/2}\rVert\mid(A_0, A_1, \ldots, A_n)\in \Pi(A)\}.
\] 
\end{lemma}

We fix $Q\in \cP_{n, 1}$. Recall that MacDonald gave $\nu(Q)=\frac{1}{2}\sec\frac{\pi}{n+2}$ \cite{M1}.

\begin{theorem}\label{t1}
If $M\in \cM_n$ and there is a nonzero projection $P\in \cM_n$ such that $PMP=M$ and $M^*M\geq P$, then $\nu(M)\geq \nu(Q)$.
\end{theorem}
\begin{proof}
Since $Q$ is of rank $1$, we obtain
\[
\Pi(Q) =\{(c_0Q, c_1Q, \ldots, c_nQ)\mid 0=c_0\leq c_1\leq \cdots\leq c_n=1\}.
\]
Hence Lemma \ref{l1} implies
\[
\nu(Q)=\inf\{\max_{1\leq k\leq n} \sqrt{c_k(1-c_{k-1})}\mid 0=c_0\leq c_1\leq \cdots\leq c_n=1\}.
\]

Let $(A_0, A_1, \ldots, A_n)\in \Pi(P)$ and set $a_k = \lVert A_k\rVert$ for each $k$. 
Since $0=A_0\leq A_1\leq \cdots\leq A_n=P$, the inequality $0=a_0\leq a_1\leq \cdots\leq a_n= 1$ holds. 
For each $k$, we have
\[
\begin{split}
\lVert (P-A_{k-1})^{1/2}(MA_kM^*)^{1/2}\rVert^2&=\lVert (MA_kM^*)^{1/2}(P-A_{k-1})(MA_kM^*)^{1/2}\rVert\\
&\geq \lVert (MA_kM^*)^{1/2}(1-a_{k-1})P(MA_kM^*)^{1/2}\rVert\\
&=  (1-a_{k-1})\lVert MA_kM^*\rVert.
\end{split}
\]
Moreover, the assumption $M^*M\geq P$ implies 
\[
\begin{split}
\lVert MA_kM^*\rVert &= \lVert (MA_k^{1/2})(MA_k^{1/2})^*\rVert = \lVert (MA_k^{1/2})^*(MA_k^{1/2})\rVert = \lVert A_k^{1/2}M^*MA_k^{1/2}\rVert\\
&\geq \lVert A_k^{1/2}PA_k^{1/2}\rVert =  \lVert A_k\rVert =a_k.
\end{split}
\]
It follows that
$\max_{1\leq k\leq n} \lVert (P-A_{k-1})^{1/2}(MA_kM^*)^{1/2}\rVert\geq  \max_{1\leq k\leq n}\sqrt{a_k(1-a_{k-1})}$.
Lemma \ref{l2} with $A=P$ and $X=M$ together with the preceding paragraph leads us to $\nu(M)=\nu(P^{1/2}MP^{1/2})\geq \nu(Q)$.
\end{proof}

If $\rank P\, (=\rank M)=m$ in the above proof, then $a_{n-m+1}=1$ since $\rank (A_k-A_{k-1})\leq 1$ for each $k$. 
Thus we may actually get $\nu(M)\geq \nu(R)$ with $R\in \cP_{n-m+1, 1}$, or equivalently, $\nu(M)\geq \frac{1}{2}\sec\frac{\pi}{n-m+3}$.

Note that the assumption on $M$ in Theorem \ref{t1} is satisfied if and only if $M$ is unitarily equivalent to a block matrix of the form $\begin{pmatrix}M_0&0\\0&0\end{pmatrix}$ for some $M_0\in \cM_m$, $1\leq m\leq n$, with $M_0^*M_0\geq I_m$. 
In particular, every nonzero normal matrix whose spectrum is contained in $\{0\}\cup \{z\in \mathbb{C}\mid \lvert z\rvert\geq 1\}$ satisfies the assumption of Theorem \ref{t1}.
\medskip

Let $H$ be a Hilbert space.
Let $\Phi_n(H)$ denote the set of $(n+1)$-tuples of projections $(P_0, P_1, \ldots, P_n)$ satisfying $0=P_0\leq P_1\leq \cdots \leq P_n=I\in \cB(H)$. 
As in the case of $\cM_n$, for $A\in \cB(H)$, the Arveson-type distance formula \cite[Lemma]{P} shows
\[
\nu_n(A) = \inf\{\max_{1\leq k\leq n}\lVert P_{k-1}^\perp AP_k\rVert\mid (P_0, P_1, \ldots, P_n)\in \Phi_n(H)\}.
\] 
Using this and imitating the first half of the proof of Lemma \ref{l1}, we obtain 
\begin{lemma}
Let $0\leq A\in \cB(H)$ and $X\in \cB(H)$. Let $\Phi_n(A)$ denote the set of $(n+1)$-tuples of operators $(A_0, A_1, \ldots, A_n)$ with $0=A_0\leq A_1\leq \cdots\leq A_n=A$.
Then
\[
\nu_n(A^{1/2}XA^{1/2})\geq \inf\{\max_{1\leq k\leq n}\lVert (A-A_{k-1})^{1/2}(XA_kX^*)^{1/2}\rVert\mid(A_0, A_1, \ldots, A_n)\in \Phi_n(A)\}.
\] 
\end{lemma}

Let $\cX(H)$ denote the set of operators $M\in \cB(H)$ with the property that there is a nonzero projection $P\in \cB(H)$ satisfying $PMP=M$ and $M^*M\geq P$.
Imitating the proof of Theorem \ref{t1}, we obtain 

\begin{theorem}\label{t2}
If $M\in \cX(H)$, then $\nu_n(M)\geq \nu(Q)$.
\end{theorem}

\begin{corollary}
If $\dim H\geq n$, then the distance $\delta$ from $\cX(H)$ to $\{N\in \cB(H)\mid N^n=0\}$ is $\nu(Q)$.
\end{corollary}
\begin{proof}
The preceding theorem implies $\delta\geq \nu(Q)$. 
Note that $\nu(Q)=\nu_n(Q)$ because $N^n=0$ for every nilpotent $N\in \cM_n$.
(Here, $\nu_n(Q)$ is the distance from $Q$ to $\{N\in \cM_n\mid N^n=0\}$.)
Take a linear isometry $V\colon \bC^n\to H$. 
Then the map $X\mapsto VXV^*$ is a linear isometry from $\cM_n$ into $\cB(H)$ that preserves the $^*$-structure and product. 
Thus $\nu(Q)=\nu_n(Q)\geq \nu_n(VQV^*)$. 
Since $VQV^*$ is a nonzero projection, it clearly belongs to $\cX(H)$. 
Therefore, $\nu_n(VQV^*)\geq \delta$.
\end{proof}


\medskip


\begin{thebibliography}{0}
\bibitem{A} W. Arveson, Interpolation problems in nest algebras. \emph{J. Funct. Anal.} \textbf{20} (1975), no. 3, 208--233.
\bibitem{C} Z. Cramer, The Distance from a Rank $n-1$ Projection to the Nilpotent Operators on $\mathbb{C}^n$. \emph{Canad. Math. Bull.} \textbf{64} (2021), no. 1, 54--74.
\bibitem{Hed} J.H. Hedlund. Limits of nilpotent and quasinilpotent operators. \emph{Michigan Math. J.} \textbf{19} (1972), 249--255.
\bibitem{Her5} D.A. Herrero, Approximation of Hilbert space operators. Vol. I. Second edition, Pitman Res. Notes Math. Ser. 224, 
1989. xii+332 pp.
\bibitem{M1} G.W. MacDonald, Distance from projections to nilpotents. \emph{Canad. J. Math.} \textbf{47} (1995), no. 4, 841--851.
\bibitem{M2} G.W. MacDonald, Distance from idempotents to nilpotents. \emph{Canad. J. Math.} \textbf{59} (2007), no. 3, 638--657.
\bibitem{P} S.C. Power, The distance to upper triangular operators. \emph{Math. Proc. Cambridge Philos. Soc.} \textbf{88} (1980), no. 2, 327--329.
\end{thebibliography}
\end{document}